\documentclass[a4paper]{amsart}

\usepackage{amsmath,amssymb,amsthm,rotating,lscape,longtable}
\usepackage[silent,nohug,tight]{diagrams}

\swapnumbers
\newtheorem{lemma}{Lemma}[section]
\newtheorem{theorem}[lemma]{Theorem}
\newtheorem{corollary}[lemma]{Corollary}
\newtheorem{proposition}[lemma]{Proposition}

\theoremstyle{definition}
\newtheorem{definition}[lemma]{Definition}

\newtheorem{defconst}[lemma]{Definition/Construction}

\newtheorem{remark}[lemma]{Remark}
\newtheorem{remarks}[lemma]{Remarks}

\newcommand{\ds}[1]{\!\!\left<#1\right>}
\newcommand{\tw}[2]{t^{\left<#1|#2\right>}}
\newcommand{\env}{^{\rm e}}
\DeclareMathOperator{\HH}{HH}
\DeclareMathOperator{\card}{card}
\DeclareMathOperator{\Tot}{Tot}

\newcommand{\then}{\Rightarrow}

\DeclareMathAlphabet{\mathpzc}{OT1}{pzc}{m}{it}

\DeclareMathOperator{\modules}{-mod} \renewcommand{\mod}{\modules}
 
\DeclareMathOperator{\proj}{-proj}

\DeclareMathOperator{\stabmod}{-\underline{mod}}

\DeclareMathOperator{\End}{End} \DeclareMathOperator{\Hom}{Hom}
\DeclareMathOperator{\Ext}{Ext}

\DeclareMathOperator{\Rad}{Rad}

\DeclareMathOperator{\Ker}{Ker}

\DeclareMathOperator{\gld}{gld}

\DeclareMathOperator{\repdim}{repdim}

\DeclareMathOperator{\op}{op}

\DeclareMathOperator{\cx}{cx}

\newcommand{\gr}{_{\rm gr}}

\DeclareMathOperator{\grHom}{grHom}

\newarrow{Dash}{}{dash}{}{dash}{}
\newarrow{Lembed} >----

\def\clap#1{\hbox to 0pt{\hss#1\hss}}

\title{Cohomology of twisted tensor products}

\author{Petter Andreas Bergh}
\author{Steffen Oppermann}
\address{Institutt for matematiske fag\\ NTNU\\ 7491 Trondheim\\ Norway}
\email{bergh@math.ntnu.no}
\email{Steffen.Oppermann@math.ntnu.no}

\subjclass[2000]{16E40, 16S80, 16U80, 81R50}

\keywords{Twisted tensor products, Hochschild cohomology, quantum
complete intersections}

\begin{document}

\begin{abstract}
It is well known that the cohomology of a tensor product is
essentially the tensor product of the cohomologies. We look at
twisted tensor products, and investigate to which extend this is
still true. We give an explicit description of the $\Ext$-algebra of
the tensor product of two modules, and under certain additional
conditions, describe an essential part of the Hochschild cohomology
ring of a twisted tensor product. As an application, we characterize
precisely when the cohomology groups over a quantum complete
intersection are finitely generated over the Hochschild cohomology
ring. Moreover, both for quantum complete intersections and in
related cases we obtain a lower bound for the representation
dimension of the algebra.
\end{abstract}

\maketitle

\section{Introduction}
Given a field $k$ and two $k$-algebras $\Lambda$ and $\Gamma$, one
may look at their tensor product $\Lambda \otimes_k \Gamma$. This is
an algebra where multiplication is done componentwise. In other
words, we use the multiplications in $\Lambda$ and $\Gamma$, and
define elements of $\Lambda$ and $\Gamma$ to commute with one
another. Given a $\Lambda$-module $M$ and a $\Gamma$-module $N$, it
is well known that
\begin{eqnarray*}
\Ext^*_{\Lambda \otimes_k \Gamma}(M \otimes_k N) & = &
\Ext^*_{\Lambda}(M) \overline{\otimes}_k \Ext^*_{\Gamma}(N), \\
\HH^*(\Lambda \otimes_k \Gamma) & = & \HH^*(\Lambda)
\overline{\otimes}_k \HH^*(\Gamma),
\end{eqnarray*}
where $\overline{\otimes}$ is the usual tensor product, but with
elements of odd degree anticommuting (and where $\HH^*$ denotes the
Hochschild cohomology ring).

In this paper we shall study graded algebras and twisted tensor
products. That is, for two graded algebras $\Lambda$ and $\Gamma$,
we give their tensor product an algebra structure by defining
elements from $\Lambda$ and $\Gamma$ to commute up to certain
scalars, depending on the degrees of the elements. We denote these
twisted tensor products by $\Lambda \otimes_k^t \Gamma$. Examples of
algebras obtained in this way are quantum exterior algebras (see
\cite{Bergh}, \cite{Beqext}, \cite{ESqext}), and, more generally,
quantum complete intersections (see \cite{BEH}, \cite{BeEqci},
\cite{BeO}).

The first main result of this paper
(Theorem~\ref{theorem.moduleext}) shows that the first formula above
holds for twisted tensor products. More precisely, we may make the
identification
\[ \Ext^*_{\Lambda \otimes_k^t \Gamma}(M \otimes_k^t N) =
\Ext^*_{\Lambda}(M) \otimes_k^{\tilde{t}} \Ext^*_{\Gamma}(N), \]
where the twist on the right-hand side is the combination of the
twist we started with and the sign which already occurred in the
classical case. This formula allows us to give an explicit
description of the $\Ext$-algebra of the simple module over a
quantum complete intersection in Theorem~\ref{theorem.extring}. As
for the second formula above, we shall see
(Remark~\ref{remark.hhnotpos}) that in general it does not carry
over to twisted tensor products. However, in
Theorem~\ref{theorem.hochschild} we show that the Hochschild
cohomology ring of a twisted tensor product contains a subalgebra,
which is the twisted tensor product of corresponding subalgebras of
the Hochschild cohomology rings of the factors. Under certain
additional conditions, we show that these subalgebras are big enough
to contain all the information on finite generation and complexity
(Corollary~\ref{corollary.finindexok}). When finite generation
holds, we may use these subalgebras and a result from \cite{Bergh}
to find a lower bound for the representation dimension of the
twisted tensor product.

In the final section we apply these results to quantum complete intersections. In
particular, we show that the cohomology groups of such an algebra
are all finitely generated over the Hochschild cohomology ring if
and only if all the commutator parameters are roots of unity
(Theorem~\ref{theorem.mainqci}). This allows us to give a lower
bound for the representation dimension of these algebras
(Corollary~\ref{corollary.lowerrepdimqci}), thus generalizing the
result of \cite{BeO}.

\section{Notation}

Throughout this paper, we fix a field $k$. All algebras considered
are assumed to be associative $k$-algebras.

\begin{definition}
Let $A$ be an abelian group. An \emph{$A$-graded} algebra is an
algebra $\Lambda$ together with a decomposition $\Lambda = \oplus_{a
\in A} \Lambda_a$ as $k$-vector spaces, such that $\Lambda_a \cdot
\Lambda_{a^{\prime}} \subseteq \Lambda_{a + a^{\prime}}$. A module
$M$ over such a graded algebra $\Lambda$ is a \emph{graded module}
if it has a decomposition $M = \oplus_{a \in A} M_a$ as $k$-vector
spaces, such that $\Lambda_a \cdot M_{a^{\prime}} \subseteq M_{a +
a^{\prime}}$. We denote the category of finitely generated graded
$\Lambda$-modules by $\Lambda\mod\gr$.
\end{definition}

Let $\Lambda$, $A$ and $M$ be as above. We denote the degree of
homogeneous elements $\lambda \in \Lambda$ and $m \in M$ by
$|\lambda|$ and $|m|$, respectively. For an element $a \in A$ we
denote by $M\ds{a}$ the shift of $M$ having the same
$\Lambda$-module structure as $M$, but with $M\ds{a}_{a^{\prime}}
= M_{a^{\prime} - a}$. Now let $M^{\prime}$ be another graded
$\Lambda$-module. To distinguish between graded and ungraded
morphisms, we denote the set of all $\Lambda$-morphisms from $M$
to $M^{\prime}$ by $\Hom_{\Lambda}(M, M^{\prime})$, and the set of
degree preserving morphisms by $\grHom_{\Lambda}(M, M^{\prime})$.
With this notation we obtain a decomposition
\[ \Hom_{\Lambda}(M, M^{\prime}) = \oplus_{a \in A} \grHom_{\Lambda}(M, M^{\prime}\ds{a}). \]
Setting $\Hom_{\Lambda}(M, M^{\prime})_{a} = \grHom_{\Lambda}(M,
M^{\prime}\ds{a})$ turns $\End_{\Lambda}(M)$ and
$\End_{\Lambda}(M^{\prime})$ into $A$-graded algebras, and
$\Hom(M,M^{\prime})$ into a graded
$\End_{\Lambda}(M)$-$\End_{\Lambda}(M^{\prime})$ bimodule. Since $M$
has a graded projective resolution $\mathbb{P}$, we can also define
$\Ext^{i,a}_{\Lambda}(M, M^{\prime}) \stackrel{\text{def}}{=} H^i
(\grHom(\mathbb{P}, M^{\prime}\ds{a}))$. It follows that
$\Ext^*_{\Lambda}(M, M)$ and $\Ext^*_{\Lambda}(M^{\prime},
M^{\prime})$ are $(\mathbb{Z} \oplus A)$-graded algebras, and that
$\Ext^*_{\Lambda}(M, M^{\prime})$ is a graded $\Ext^*_{\Lambda}(M,
M)$-$\Ext^*_{\Lambda}(M^{\prime}, M^{\prime})$ bimodule.

Our main objects of study in this paper are twisted tensor products
of two graded algebras, a concept we now define.

\begin{defconst}
Let $A$ and $B$ be abelian groups, let $\Lambda$ be an $A$-graded
algebra and $\Gamma$ a $B$-graded algebra. Let $t: A
\otimes_{\mathbb{Z}} B \rTo k^{\times}$ be a homomorphism of abelian
groups, where $k^{\times}$ denotes the multiplicative group of
nonzero elements in $k$. We write $\tw{a}{b} = t(a \otimes b)$, and,
by abuse of notation, also $\tw{\lambda}{\gamma} = \tw{| \lambda
|}{| \gamma |}$ for homogeneous elements $\lambda \in \Lambda$ and
$\gamma \in \Gamma$. The ($t$-)\emph{twisted tensor product} of
$\Lambda$ and $\Gamma$ is the algebra $\Lambda \otimes_k^t \Gamma$
defined by
\begin{align*}
\Lambda \otimes_k^t \Gamma & = \Lambda \otimes_k \Gamma
\hspace{2mm} \text{ as $k$-vector spaces,} \\
(\lambda \otimes \gamma) \cdot^t (\lambda^{\prime} \otimes
\gamma^{\prime}) & \stackrel{\text{def}}{=}
\tw{\lambda^{\prime}}{\gamma} \lambda \lambda^{\prime} \otimes
\gamma \gamma^{\prime},
\end{align*}
where $\lambda, \lambda' \in \Lambda$ and $\gamma, \gamma' \in
\Gamma$ are homogeneous elements.
\end{defconst}

A straightforward calculation shows that this is indeed a well
defined algebra. By defining $(\Lambda \otimes_k^t \Gamma)_{a,b}$ to
be $\Lambda_a \otimes_k \Gamma_b$, this algebra becomes $(A \oplus
B)$-graded.

We now define what it means for an algebra to have finitely
generated cohomolgy.

\begin{definition}
Let $\Lambda$ be an algebra. A commutative \emph{ring of cohomolgy
operators} is a commutative $\mathbb{Z}$-graded $k$-algebra $H$
together with graded $k$-algebra morphisms $\phi_M: H \rTo
\Ext^*_{\Lambda}(M,M)$, for every $M \in \Lambda\mod$, such that for
every pair $M,M^{\prime} \in \Lambda\mod$ the induced $H$-module
structures on $\Ext^*_{\Lambda}(M, M^{\prime})$ via $\phi_M$ and
$\phi_{M^{\prime}}$ coincide. If $A$ is an abelian group and
$\Lambda$ is $A$-graded, then we require that $H$ be a $(\mathbb{Z}
\oplus A)$-graded algebra, and that the maps $\phi_M$ are morphisms
of $(\mathbb{Z} \oplus A)$-graded algebras.
\end{definition}

The main example of such a ring $H$ is the even Hochschild
cohomology ring of an algebra. Namely, by \cite{Yoneda} the
Hochschild cohomology ring is graded commutative, so its even part
is commutative. Note that whenever an algebra is graded, then so is
its Hochschild cohomology ring, and its even part is a commutative
ring of graded cohomology operators.

\begin{definition}
An algebra $\Lambda$ satisfies the \emph{finite generation
hypothesis} {\bf{Fg}} if it has a commutative ring of operators $H$
which is Noetherian and of finite type (i.e.\ $\dim_k H_i < \infty$
for all $i$), and such that for any $M, M^{\prime} \in \Lambda\mod$
the $H$-module $\Ext^*_{\Lambda}(M, M^{\prime})$ is finitely
generated.
\end{definition}

We end this section with some remarks concerning finite generation
of cohomology.

\begin{remarks} \label{remark.simplesenough}
(i) Assume that $\Lambda$ is a finite dimensional algebra, and let
$H$ be a commutative Noetherian ring of cohomology operators. Then
all $\Ext^*_{\Lambda}(M, M^{\prime})$ are finitely generated over
$H$ if and only if $\Ext^*_{\Lambda}(\Lambda / \Rad \Lambda, \Lambda
/ \Rad \Lambda)$ is. This follows from an induction argument on the
length of $M$ and $M'$.

(ii) By \cite[Proposition~5.7]{SoSupportVar} the following are
equivalent for an algebra $\Lambda$.
\begin{enumerate}
\item $\Lambda$ satisfies {\bf{Fg}} with respect to its even
Hochschild cohomology ring $\HH^{2*} ( \Lambda )$,
 \item $\Lambda$ satisfies {\bf{Fg}}
with respect to some subalgebra of its even Hochschild cohomology
ring.
\end{enumerate}
\end{remarks}

\section{Tensor products of graded modules}

Throughout this section, we fix two abelian groups $A$ and $B$,
together with an $A$-graded algebra $\Lambda$ and a $B$-graded
algebra $\Gamma$. Moreover, we fix a homomorphism $t: A
\otimes_{\mathbb{Z}} B \rTo k^{\times}$ of abelian groups. Given a graded $\Lambda$-module and a graded
$\Gamma$-module, we construct a $\Lambda \otimes_k^t
\Gamma$-module, and study homomorphisms and extensions between
such modules.

\begin{defconst}
Given modules $M \in \Lambda\mod\gr$ and $N \in \Gamma\mod\gr$,
the tensor product $M \otimes_k N$ becomes a graded $\Lambda
\otimes_k^t \Gamma$-module by defining
\[ (\lambda \otimes \gamma) \cdot (m \otimes n)
\stackrel{\text{def}}{=} \tw{m}{\gamma} \lambda m \otimes \gamma
n.\] We denote this module by $M \otimes_k^t N$, its grading is
given by $(M \otimes_k^t N)_{a,b} = M_a \otimes_k N_b$.
\end{defconst}

We now prove some elementary results on these tensor products, the
first of which shows that the tensor product of two shifted
modules is the shifted tensor product.

\begin{lemma} \label{lemma.moduletensor.shiftsok}
Given modules $M \in \Lambda\mod\gr$ and $N \in \Gamma\mod\gr$,
the graded $\Lambda \otimes_k^t \Gamma$-modules $M\ds{a}
\otimes_k^t N\ds{b}$ and $(M \otimes_k^t N)\ds{a,b}$ are
isomorphic via the map
\begin{align*}
M\ds{a} \otimes_k^t N\ds{b} & \rTo (M \otimes_k^t N)\ds{a,b} \\
m \otimes n & \rMapsto \tw{a}{n} m \otimes n
\end{align*}
\end{lemma}

\begin{proof}
The given map is clearly bijective, and it is straightforward to
verify that it is a homomorphism.
\end{proof}

The following lemma shows that the tensor product of projective
modules is again projective. Given a graded algebra $\Delta$, we
denote by $\Delta\proj$ the category of finitely generated
projective $\Delta$-modules, and by $\Delta\proj\gr$ the category
of finitely generated graded projective $\Delta$-modules.

\begin{lemma} \label{lemma.tensorproj}
Given modules $P \in \Lambda\proj\gr$ and $Q \in \Gamma\proj\gr$,
the tensor product $P \otimes_k^t Q$ is a graded projective
$\Lambda \otimes_k^t \Gamma$-module.
\end{lemma}

\begin{proof}
By Lemma~\ref{lemma.moduletensor.shiftsok} we only have to consider
the case $P = \Lambda$ and $Q = \Gamma$. In this case $P \otimes_k^t
Q = \Lambda \otimes_k^t \Gamma$, so the lemma holds.
\end{proof}

As the following result shows, the tensor product of morphism
spaces is the morphism space of tensor products.

\begin{lemma} \label{lemma.tensorprodmorph}
Given modules $M, M^{\prime} \in \Lambda\mod\gr$ and $N,
N^{\prime} \in \Gamma\mod\gr$, the natural map
\[ \grHom_{\Lambda}(M, M^{\prime}) \otimes_k \grHom_{\Gamma}(N,
N^{\prime}) {\rTo} \grHom_{\Lambda \otimes_k^t \Gamma}(M
\otimes_k^t N, M^{\prime} \otimes_k^t N^{\prime}) \] is an
isomorphism.
\end{lemma}

\begin{proof}
If $M = \Lambda\ds{a}$ and $N = \Gamma\ds{b}$ for some $a \in A$
and $b \in B$, then
\begin{align*}
\grHom_{\Lambda}(\Lambda\ds{a}, M^{\prime}) \otimes_k
\grHom_{\Gamma}(\Gamma\ds{b}, N^{\prime}) & = M^{\prime}_{-a}
\otimes_k N^{\prime}_{-b} \\
& = (M^{\prime} \otimes_k^t N^{\prime})_{-a,-b} \\
& = \grHom((\Lambda \otimes_k^t \Gamma)\ds{a,b}, M^{\prime} \otimes_k^t N^{\prime}) \\
& = \grHom(\Lambda\ds{a} \otimes_k^t \Gamma\ds{b}), M^{\prime}
\otimes_k^t N^{\prime}).
\end{align*}
Now note that both sides commute with cokernels in the $M$ and $N$ position.
\end{proof}

Note that given degree $a$ and $b$ morphisms $\varphi: M \rTo
M^{\prime}$ and $\psi: N \to N^{\prime}$, we obtain a degree
$(a,b)$-morphism
\[ \varphi \otimes \psi: M \otimes N \rTo M^{\prime} \otimes N^{\prime} \]
by composing the maps from Lemmas~\ref{lemma.tensorprodmorph} and
\ref{lemma.moduletensor.shiftsok}. Explicitly, the map is given by
\[ (m \otimes n) \cdot (\varphi \otimes \psi) = \tw{\varphi}{n} m \cdot \varphi \otimes n \cdot \psi \]
(we think of a module as a right module over its endomorphism ring).
By applying this to the situation $M = M^{\prime}$ and $N =
N^{\prime}$, we obtain the following result, showing that the
endomorphism ring of a tensor product is the tensor product of the
endomorphism rings.

\begin{lemma} \label{lemma.twistedendo}
Let $M \in \Lambda\mod\gr$ and $N \in \Gamma\mod\gr$. Then
\[ \End_{\Lambda \otimes_k^t \Gamma}(M \otimes_k^t N) = \End_{\Lambda}(M)
\otimes_k^t \End_{\Gamma}(N).\]
\end{lemma}

As for projective resolutions, the behavior is also as expected.
Namely, the following result shows that the tensor product of two
projective resolutions is again a projective resolution.

\begin{lemma} \label{lemma.moduletensor.projresol}
Given modules $M \in \Lambda\mod\gr$ and $N \in \Gamma\mod\gr$
with graded projective resolutions
\begin{align*}
& \mathbb{P}: \cdots \rTo P_i \rTo \cdots \rTo P_1 \rTo P_0 \rTo M
\rTo 0, \\
& \mathbb{Q}: \cdots \rTo Q_i \rTo \cdots \rTo Q_1 \rTo Q_0 \rTo N
\rTo 0,
\end{align*}
the total complex of $\mathbb{P} \otimes_k^t \mathbb{Q}$ is a
graded projective resolution of $M \otimes_k^t N$.
\end{lemma}

\begin{proof}
By Lemma~\ref{lemma.tensorproj} all the terms of the total complex
$\Tot(\mathbb{P} \otimes_k^t \mathbb{Q})$ of $\mathbb{P}
\otimes_k^t \mathbb{Q}$ are projective. Moreover, since $k$ is a
field $\Tot(\mathbb{P} \otimes_k^t \mathbb{Q})$ is exact.
\end{proof}

We are now ready to prove the main result of this section. It
shows that the $\Ext$-algebra of a tensor product is the tensor
product of the $\Ext$-algebras.

\begin{theorem} \label{theorem.moduleext}
If $M, M^{\prime} \in \Lambda\mod\gr$ and $N, N^{\prime} \in
\Gamma\mod\gr$ are modules, then
\[ \Ext^*_{\Lambda \otimes_k^t \Gamma}(M \otimes_k^t N, M \otimes_k^t N) =
\Ext^*_{\Lambda}(M, M) \otimes_k^{\tilde{t}} \Ext^*_{\Gamma}(N, N), \]
with $\tilde{t}((i,a),(j,b)) = (-1)^{ij} \tw{a}{b}$.
Moreover
\[ \Ext^*_{\Lambda \otimes_k^t \Gamma}(M \otimes_k^t N, M^{\prime}
\otimes_k^t N^{\prime}) = \Ext^*_{\Lambda}(M, M^{\prime})
\otimes_k^{\tilde{t}} \Ext^*_{\Gamma}(N, N^{\prime}) \] as
$\Ext^*_{\Lambda}(M, M) \otimes_k^t \Ext^*_{\Gamma}(N, N) -
\Ext^*_{\Lambda} (M^{\prime}, M^{\prime}) \otimes_k^t
\Ext^*_{\Gamma}(N^{\prime}, N^{\prime})$ bimodule.
\end{theorem}

\begin{proof}
Let $\mathbb{P}$ and $\mathbb{Q}$ be graded projective resolutions
of $M$ and $N$ receptively. Then by
Lemma~\ref{lemma.moduletensor.projresol} $\Tot(\mathbb{P}
\otimes_k^t \mathbb{Q})$ is a projective resolution of $M
\otimes_k^t N$, and therefore
\begin{align*}
\Ext^*_{\Lambda \otimes_k^t \Gamma}(M \otimes_k^t N, M'
\otimes_k^t N') & = H^*(\Hom_{\Lambda \otimes_k^t
\Gamma}(\Tot(\mathbb{P} \otimes_k^t \mathbb{Q}), M^{\prime}
\otimes_k^t N^{\prime}) \\ & = H^*(\Tot(\Hom_{\Lambda \otimes_k^t
\Gamma}(\mathbb{P} \otimes_k^t \mathbb{Q}), M^{\prime} \otimes_k^t
N^{\prime}) \\ & = H^*(\Tot(\Hom_{\Lambda}(\mathbb{P}, M^{\prime})
\otimes_k^t \Hom_{\Gamma}(\mathbb{Q}, N^{\prime}))) \\ & =
\Ext^*_{\Lambda}(M, M^{\prime}) \otimes_k^{\tilde{t}}
\Ext^*_{\Gamma}(N, N^{\prime}).
\end{align*}
Here the third equality holds by
Lemma~\ref{lemma.tensorprodmorph}, whereas the final one holds
since $k$ is a field. The multiplication is induced by the
multiplication of morphisms in Lemma~\ref{lemma.twistedendo}, with
the additional signs needed because of the signs added when
passing from the double complex to its total complex.
\end{proof}

We end this section with the following result, which was shown in
\cite{MR1749892} for untwisted tensor products (in which case we
may forget about the grading). It will help us find upper bounds
for the representation dimension of twisted tensor products. Given
an algebra $\Delta$, we denote by $\gld \Delta$ its global
dimension.

\begin{proposition} \label{proposition.repdimadd}
Let $M \in \Lambda\mod\gr$ and $N \in \Gamma\mod\gr$ be graded
modules, such that $M$ generates and cogenerates $\Lambda\mod$,
and such that $N$ generates and cogenerates $\Gamma \mod$. Then $M
\otimes_k^t N$ is a generator-cogenerator of $\Lambda \otimes_k^t
\Gamma\mod$, and $\gld \End_{\Lambda \otimes_k^t \Gamma} (M
\otimes_k^t N) = \gld \End_{\Lambda}(M) + \gld \End_{\Gamma}(N)$.
\end{proposition}

\section{Tensor products of bimodules}

Throughout this section, we keep the notation from the last
section. That is, we fix two abelian groups $A$ and $B$, together
with an $A$-graded algebra $\Lambda$ and a $B$-graded algebra
$\Gamma$. Moreover, we fix a homomorphism $t: A
\otimes_{\mathbb{Z}} B \rTo k^{\times}$ of abelian groups.
Given an algebra $\Delta$, we denote by $\Delta\env$ its
enveloping algebra $\Delta \otimes_k \Delta^{\op}$. Note that if
$\Delta$ is $G$-graded, where $G$ is some abelian group, then so
is $\Delta\env$, and $\Delta$ is a graded $\Delta\env$-module.

\begin{defconst}
Given modules $X \in \Lambda\env\mod\gr$ and $Y \in
\Gamma\env\mod\gr$, the tensor product $X \otimes_k Y$ becomes a
graded $(\Lambda \otimes_k^t \Gamma)\env$-module by defining
$$(\lambda \otimes \gamma) (x \otimes y) (\lambda^{\prime} \otimes
\gamma^{\prime}) \stackrel{\text{def}}{=} \tw{x}{\gamma}
\tw{\lambda^{\prime}}{y} \tw{\lambda^{\prime}}{\gamma} \lambda x
\lambda^{\prime} \otimes \gamma y \gamma^{\prime}.$$ We denote
this bimodule by $X \otimes_k^t Y$.
\end{defconst}

\begin{remark}
In general the graded $(\Lambda \otimes_k^t \Gamma)\env$-modules
$X\ds{a} \otimes_k^t Y\ds{b}$ and $(X \otimes_k^t Y) \ds{a,b}$ are
not isomorphic.
\end{remark}

The following results are analogues of Lemmas~\ref{lemma.tensorproj},
\ref{lemma.tensorprodmorph} and
\ref{lemma.moduletensor.projresol}. We prove only the first result, as the proofs of the other two results are more or less the same as those of Lemmas~\ref{lemma.tensorprodmorph} and
\ref{lemma.moduletensor.projresol}.

\begin{lemma}
Given modules $X \in \Lambda\env\proj\gr$ and $Y \in \Gamma\env\proj\gr$, the tensor product $X \otimes_k^t Y$ is a graded projective $(\Lambda \otimes_k^t \Gamma)\env$-module.
\end{lemma}

\begin{proof}
It suffices to show that $\Lambda\env\ds{a} \otimes_k^t
\Gamma\env\ds{b}$ is graded projective for any $a \in A$ and $b \in B$. This can be seen by noting that the map
\begin{align*}
(\Lambda \otimes_k^t \Gamma)\env\ds{a,b} & \rTo \Lambda\env\ds{a}
\otimes_k^t \Gamma\env\ds{b} \\ (l \otimes g) \otimes (l^{\prime}
\otimes g^{\prime}) & \rMapsto \tw{l^{\prime}}{g} \tw{a}{g}
\tw{l^{\prime}}{b} (l \otimes l^{\prime}) \otimes (g \otimes
g^{\prime})
\end{align*}
 is an isomorphism of graded $( \Lambda \otimes_k^t \Gamma )\env$-modules.
\end{proof}

\begin{lemma}
Given modules $X, X^{\prime} \in \Lambda\env\mod\gr$ and $Y, Y^{\prime} \in
\Gamma\env\mod\gr$, the natural map
\[ \grHom_{\Lambda\env}(X, X^{\prime}) \otimes_k
\grHom_{\Gamma\env}(Y, Y^{\prime}) {\rTo} \grHom_{(\Lambda
\otimes_k^t \Gamma)\env}(X \otimes_k^t Y, X^{\prime} \otimes_k^t
Y^{\prime}) \] is an isomophism.
\end{lemma}

\begin{lemma}
Given modules $X \in \Lambda\env\mod\gr$ and $Y \in \Gamma\env\mod\gr$ with graded projective bimodule resolutions
\begin{align*}
& \mathbb{P}: \cdots \rTo P_i \rTo \cdots \rTo P_1 \rTo P_0 \rTo X \rTo 0, \\
& \mathbb{Q}: \cdots \rTo Q_i \rTo \cdots \rTo Q_1 \rTo Q_0 \rTo Y \rTo 0,
\end{align*}
the total complex of $\mathbb{P} \otimes_k^t \mathbb{Q}$ is a
graded projective bimodule resolution of $X \otimes_k^t Y$.
\end{lemma}

Now note that for a fixed $b \in B$ the map $t$ induces a morphism
$\tw{-}{b}: A \rTo k^{\times}$ (and similarly for a fixed $a \in A$).
With this notation, we make the following observation.

\begin{lemma} \label{lemma.bimodshiftok}
Let $a^{\prime} \in \cap_{b \in B} \Ker \tw{-}{b} \leq A$ and
$b^{\prime} \in \cap_{a \in A} \Ker \tw{a}{-} \leq B$. Then the map
\begin{align*}
\Lambda\ds{a^{\prime}} \otimes_k^t \Gamma\ds{b^{\prime}} & \rTo
(\Lambda \otimes_k^t \Gamma)\ds{a,b}\\ \lambda \otimes \gamma &
\rMapsto \lambda \otimes \gamma
\end{align*}
is an isomorphism of graded $(\Lambda \otimes_k^t \Gamma)\env$-modules.
\end{lemma}

Using the above notation, we now prove the main result of this
section. It shows that Hochschild cohomology commutes with twisted
tensor products, provided we only consider the graded parts
corresponding to the subgroups $\cap_{b \in B} \Ker \tw{-}{b} \leq
A$ and $\cap_{a \in A} \Ker \tw{a}{-} \leq B$.

\begin{theorem} \label{theorem.hochschild}
Let $A^{\prime} = \cap_{b \in B} \Ker \tw{-}{b} \leq A$ and
$B^{\prime} = \cap_{a \in A} \Ker \tw{a}{-} \leq B$. Then there is an
isomorphism
\[ \HH^{*,A^{\prime}}(\Lambda) \otimes_k^{(-1)^{**}}
\HH^{*,B^{\prime}}(\Gamma) \rTo \HH^{*,A^{\prime} \oplus
B^{\prime}}(\Lambda \otimes_k^t \Gamma), \] where $(-1)^{**}$
denotes the morphism mapping $((i, a^{\prime}), (j, b^{\prime}))$ to
$(-1)^{ij}$.
\end{theorem}

\begin{proof}
Let $\mathbb{P}$ and $\mathbb{Q}$ be graded bimodule projective resolutions
of $\Lambda$ and $\Gamma$,
respectively. Given $a \in A$ and $b \in B$, the same arguments as
in the proof of Theorem~\ref{theorem.moduleext} give
\begin{align*}
\HH^{*,a,b}(\Lambda \otimes_k^t \Gamma) & = H^*(\grHom_{(\Lambda
\otimes_k^t \Gamma)\env}(\Tot(\mathbb{P} \otimes_k^t \mathbb{Q},
(\Lambda \otimes_k^t \Gamma)\ds{a,b}))) \\ & =
H^*(\Tot(\Hom_{(\Lambda \otimes_k^t \Gamma)\env}(\mathbb{P}
\otimes_k^t \mathbb{Q}, (\Lambda \otimes_k^t \Gamma)\ds{a,b})))
\intertext{and} & \hspace{-1cm} H^*(\Tot(\Hom_{(\Lambda \otimes_k^t
\Gamma)\env}(\mathbb{P} \otimes_k^t \mathbb{Q}, \Lambda\ds{a}
\otimes_k^t \Gamma\ds{b}))) \\
& = H^*(\Tot( \Hom_{\Lambda\env}(\mathbb{P}, \Lambda\ds{a})
\otimes_k^t \Hom_{\Gamma\env}(\mathbb{Q}, \Gamma))) \\ & =
\HH^{*,a}(\Lambda) \otimes_k^{\tilde{t}} \HH^{*,b}(\Gamma),
\end{align*}
where $\tilde{t} = (-1)^{**} \cdot t$ as in
Theorem~\ref{theorem.moduleext}. Now if $a \in A^{\prime}$ and $b \in
B^{\prime}$, then from
Lemma~\ref{lemma.bimodshiftok} we see that we may identify
$$H^*(\Tot(\Hom_{(\Lambda \otimes_k^t \Gamma)\env}(\mathbb{P}
\otimes_k^t \mathbb{Q}, (\Lambda \otimes_k^t \Gamma)\ds{a,b})))$$
with
$$H^*(\Tot(\Hom_{(\Lambda \otimes_k^t
\Gamma)\env}(\mathbb{P} \otimes_k^t \mathbb{Q}, \Lambda\ds{a}
\otimes_k^t \Gamma\ds{b}))).$$
Finally, note that
$\HH^{*,A^{\prime}}(\Lambda) \otimes_k^{\tilde{t}}
\HH^{*,B^{\prime}}(\Gamma) = \HH^{*,A^{\prime}}(\Lambda)
\otimes_k^{(-1)^{**}} \HH^{*,B^{\prime}}(\Gamma)$, since all degrees
occurring are in the kernel of $t$.
\end{proof}

We end this section with the following corollary to Theorem~\ref{theorem.hochschild}. It shows that, given certain conditions, if $\Lambda$ and $\Gamma$ satisfy {\bf{Fg}}, then so does $\Lambda \otimes_k^t \Gamma$.

\begin{corollary} \label{corollary.finindexok}
With the same notation as in Theorem~\ref{theorem.hochschild},
assume $\Lambda$ and $\Gamma$ satisfy {\bf{Fg}} with respect to
their even Hochschild cohomolgy rings $\HH^{2*} ( \Lambda )$ and $\HH^{2*} ( \Gamma )$. Moreover, suppose $[A\!:\!A^{\prime}]$ and
$[B\!:\!B^{\prime}]$ are finite, and that $\Lambda / \Rad \Lambda$ and
$\Gamma / \Rad \Gamma$ are separable over $k$. Then $\Lambda
\otimes_k^t \Gamma$ satisfies {\bf{Fg}} with respect its even
Hochschild cohomolgy ring $\HH^{2*} ( \Lambda \otimes_k^t \Gamma )$.
\end{corollary}

\begin{proof}
\sloppy Since $[A\!:\!A^{\prime}]$ is finite, the algebra $\HH^{2*,
A}(\Lambda)$ is a finitely generated module over $\HH^{2*,
A^{\prime}}(\Lambda)$. Therefore, since $\Lambda$ satisfies
{\bf{Fg}}, we see that $\Ext^*_{\Lambda}(\Lambda / \Rad \Lambda,
\Lambda / \Rad \Lambda)$ is finitely generated over
$\HH^{2*,A^{\prime}}(\Lambda)$. The same arguments apply to
$\Gamma$, hence
$$\Ext^*_{\Lambda}(\Lambda / \Rad \Lambda, \Lambda / \Rad \Lambda) \otimes_k^t
\Ext^*_{\Gamma}(\Gamma / \Rad \Gamma, \Gamma / \Rad \Gamma)$$
is finitely generated over $\HH^{2*, A^{\prime}}(\Lambda) \otimes_k \HH^{2 *,B^{\prime}}(\Gamma)$. Then by Theorems~\ref{theorem.moduleext} and \ref{theorem.hochschild}, we see that
$$\Ext^*_{\Lambda \otimes_k^t \Gamma}(\Lambda / \Rad \Lambda \otimes_k^t \Gamma / \Rad \Gamma,
\Lambda / \Rad \Lambda \otimes_k^t \Gamma / \Rad \Gamma)$$
must be a finitely generated $\HH^{2*, A^{\prime} \oplus B^{\prime}}( \Lambda \otimes_k^t \Gamma )$-module. Finally, since $\Lambda / \Rad \Lambda$ and $\Gamma / \Rad \Gamma$ are separable over $k$, the equality
\[ \Lambda / \Rad \Lambda \otimes_k^t \Gamma / \Rad \Gamma =
( \Lambda \otimes_k^t \Gamma ) / \Rad ( \Lambda \otimes_k^t \Gamma )\] holds. The
claim now follows from Remarks~\ref{remark.simplesenough}.
\end{proof}

\section{Quantum complete intersections}

We now apply the cohomology theory of twisted tensor products to the class of finite dimensional algebras known as \emph{quantum complete intersections}. Throughout this section, fix integers $n \ge 1$ and $a_1, \dots, a_n \ge 2$, together with a nonzero element $q_{ij} \in k$ for every $1 \le i<j \le n$. We define the algebra $\Lambda$ by
$$\Lambda \stackrel{\text{def}}{=} k\left<x_1, \ldots, x_n\right>
/ ( x_i^{a_i}, x_j x_i - q_{ij} x_ix_j ),$$
a codimension $n$ quantum complete intersection in its most general form. This is a selfinjective algebra of dimension $\prod a_i$. We shall determine precisely when such an algebra satisfies {\bf{Fg}}, and consequently obtain a lower bound for its representation
dimension.

Note that $\Lambda$ is $\mathbb{Z}^n$ graded by $| x_i | \stackrel{\text{def}}{=} (0, \dots, 1, \dots 0)$, the
$i$th unit vector. In particular, we use the
$\mathbb{Z}$-grading $|x| = 1$ for the special case of a codimension one quantum complete intersection $k[x] / (x^a)$. The following
observation allows us to study the cohomology inductively, starting
with the well known case $k[x]/(x^a)$.

\begin{lemma}
Let $\Lambda^{\prime}$ be the subalgebra of $\Lambda$ generated by
$x_1, \ldots, x_{n-1}$. Then
\[ \Lambda = \Lambda^{\prime} \otimes_k^t k[x_n] / (x_n^{a_n}), \]
where $\tw{d_1, \ldots, d_{n-1}}{d_n} \stackrel{\emph{def}}{=}
\prod_{i=1}^{n-1} q_{in}^{d_i d_n}$.
\end{lemma}

As for quantum complete intersections of codimension one, that is,
truncated polynomial algebras, their cohomology is well known. We
record this in the following lemma.

\begin{lemma} \label{lemma.kxmodxn}
If $\Gamma = k[x] / (x^a)$, then
\begin{itemize}
\item[(1)] $\HH^{2*}(\Gamma) = k[x,z] / (x^a, ax^{a-1}z)$,
\item[(2)] $\Ext^*_{\Gamma}(k,k) = \left\{ \begin{matrix} k[y,z]
/(y^2 = z) & \text{ if } a = 2 \\ k[y,z] / (y^2) & \text{ if } a
\neq 2 \end{matrix} \right.$,
\end{itemize}
with $|x| = 0$, $|y| = 1$ and $|z| = 2$. In particular, the algebra
$\Gamma$ satisfies {\bf{Fg}} with respect to its even Hochschild
cohomology ring.
\end{lemma}

\begin{proof}
The first part is \cite[Theorem~3.2]{Hol}, the second part can be
read off directly from the projective resolution.
\end{proof}

Using this lemma and Theorem~\ref{theorem.moduleext}, we obtain the
following result on the $\Ext$-algebra of the simple module of a
quantum complete intersection.

\begin{theorem} \label{theorem.extring}
The $\Ext$-algebra of $k$ is given by
\[ \Ext^*_{\Lambda}(k,k) = k \left < y_1, \ldots y_n, z_1,
\ldots z_n \right > / I, \] where $I$ is the ideal in $k \left <
y_1, \ldots y_n, z_1, \ldots z_n \right >$ defined by the relations
\[ \left( \begin{matrix} y_i z_i - z_i y_i & \\ y_j y_i +
q_{ij} y_i y_j & i < j \\ y_j z_i - q_{ij}^2 z_i y_j & i < j \\ z_j
y_i - q_{ij}^2 z_i y_j & i < j \\ z_j z_i - q_{ij}^4 z_i z_j & i < j
\\ y_i^2 = z_i & a_i = 2 \\ y_i^2 & a_i \neq 2 \end{matrix} \right )
\]
\end{theorem}

\begin{remark} \label{remark.hhnotpos}
Lemma~\ref{lemma.kxmodxn} shows that if $\Gamma$ and $\Delta$ are
arbitrary algebras, then the algebra $\HH^*(\Gamma)
\otimes_k^{\widetilde{t}} \HH^*(\Delta)$ does not in general embed
into $\HH^*(\Gamma \otimes_k^t \Delta)$. Namely, the latter is
always graded commutative, whereas $\HH^*(\Gamma)
\otimes_k^{\widetilde{t}} \HH^*(\Delta)$ need not be.
\end{remark}

We are now ready to characterize precisely when a quantum complete
intersection satisfies {\bf{Fg}}.

\begin{theorem} \label{theorem.mainqci}
The following are equivalent.
\begin{enumerate}
\item $\Lambda$ satisfies {\bf{Fg}},
\item $\Lambda$ satisfies {\bf{Fg}} with respect to its even Hochschild
cohomology ring $\HH^{2*} ( \Lambda )$,
\item all the commutators $q_{ij}$ are roots of unity.
\end{enumerate}
\end{theorem}

\begin{proof}
The implication (2) $\then$ (1) is obvious, and the implication (3)
$\then$ (2) follows from Corollary~\ref{corollary.finindexok}. To
show (1) $\then$ (3), we assume that (1) holds but not (3), so there
are $i$ and $j$ such that $q_{ij}$ is not a root of unity. By (1),
the $\Ext$-algebra of $k$ is finitely generated as a module over its
center, hence so is every quotient of this ring. By factoring out
all $y_k, z_k$ with $k \not\in \{i, j\}$ and $\{y_k \mid k \in
\{i,j\} \text{ and } y_k^2 = 0\}$, we obtain a ring of the form
$k\left<r,s\right> / (sr - qrs)$, where $q$ is not a root of unity.
The center of this ring is trivial, hence the ring cannot be
finitely generated over its center, a contradiction.
\end{proof}

As a corollary, we obtain a lower bound for the representation
dimension of a quantum complete intersection. Recall that the
representation dimension of a finite dimensional algebra $\Delta$ is
defined as
$$\repdim \Delta \stackrel{\text{def}}{=} \inf \{ \gld \End_{\Delta}
(M) \},$$ where the infimum is taken over all the finitely generated
$\Delta$-modules which generate and cogenerate $\Delta\mod$.

\begin{corollary} \label{corollary.lowerrepdimqci}
Define the integer $c \ge 0$ by
$$c \stackrel{\emph{def}}{=} \max \{ \card I \mid I \subseteq \{ 1 ,
\ldots, n \} \text{ and } q_{ij} \text{ is a root of unity } \forall
i, j \in I, i<j \}.$$ Then $\repdim \Lambda \ge c+1$. In particular,
if all the commutators $q_{ij}$ are roots of unity, then $\repdim
\Lambda \ge n+1$.
\end{corollary}

In order to prove this result we need to recall some notions. Let
$\Delta$ be an algebra, and let $M$ be a finitely generated
$\Delta$-module with minimal projective resolution
$$\cdots \rTo P_2 \rTo P_1 \rTo P_0 \rTo M \rTo 0,$$
say. The \emph{complexity} of $M$, denoted $\cx M$, is defined as
$$\cx M \stackrel{\text{def}}{=} \inf \{ t \in \mathbb{N} \cup \{ 0
\} \mid \exists r \in \mathbb{R} \text{ such that } \dim_k P_n \le
rn^{t-1} \text{ for } n \gg 0 \}.$$ Now let $V$ be a positively
graded $k$-vector space of finite type, i.e.\ $\dim_k V_n < \infty$
for all $n$. The \emph{rate of growth} of $V$, denoted $\gamma (V)$,
is defined as
$$\gamma (V) \stackrel{\text{def}}{=} \inf \{ t \in \mathbb{N} \cup \{ 0
\} \mid \exists r \in \mathbb{R} \text{ such that } \dim_k V_n \le
rn^{t-1} \text{ for } n \gg 0 \}.$$ It is well known that the
complexity of a module $M$ equals $\gamma \left ( \Ext_{\Delta}^*
(M, \Delta / \Rad \Delta ) \right )$. Now suppose $\Delta$ is
selfinjective, and denote by $\Delta\stabmod$ the stable module
category of $\Delta\mod$, that is, the category obtained from
$\Delta\mod$ by factoring out all morphisms which factor through a
projective module. This is a triangulated category, and we denote by
$\dim ( \Delta \stabmod )$ its dimension, as defined in \cite{Rou}.

\begin{proof}[Proof of Corollary~\ref{corollary.lowerrepdimqci}]
Choose a subset $I$ of $\{ 1 , \ldots, n \}$ realizing the maximum
in the definition of the integer $c$, and let $\Lambda^{\prime}$ be
the subalgebra of $\Lambda$ generated by the $x_i$ with $i \in I$.
By Theorem~\ref{theorem.mainqci} the algebra $\Lambda^{\prime}$
satisfies {\bf{Fg}}, and so from \cite[Theorem~3.1]{Bergh} we see
that $\dim ( \Lambda^{\prime}\stabmod ) \geq \cx_{\Lambda'} k - 1$.
Moreover, by Theorem~\ref{theorem.extring} the complexity of $k$ as
a $\Lambda'$-module equals $\card I$, giving $\dim (
\Lambda^{\prime}\stabmod ) \geq \card I - 1$.

The forgetful functor $\Lambda\mod \rTo \Lambda^{\prime}\mod$ is
exact, dense, and maps projective $\Lambda$-modules to projective
$\Lambda^{\prime}$-modules. Therefore it induces a dense triangle
functor $\Lambda\stabmod \rTo \Lambda^{\prime}\stabmod$, and so from
\cite[Lemma~3.4]{Rou} we obtain the inequality $\dim (
\Lambda\stabmod ) \geq \dim ( \Lambda^{\prime}\stabmod )$. Finally,
by \cite[Proposition~3.7]{Rou1} the inequality $\repdim \Lambda \geq
\dim ( \Lambda\stabmod ) + 2$ holds, and the proof is complete.
\end{proof}

\begin{remark} \label{remark.2nupper}
By \cite[Theorem~3.2]{BeO} the inequality $\repdim \Lambda \leq 2n$
always holds.
\end{remark}

It was shown in \cite{MR2236768} that the representation dimension
of the truncated polynomial algebra $k[x,y] / (x^2, y^a)$ is three.
Using their construction and exactly the same proof, one can show
that the quantum complete intersection $\Gamma = k\left<x,y\right> /
(yx-qxy, x^2, y^a)$ has a generator-cogenerator $M$ which is graded
with $\gld \End_{\Gamma} (M) = 3$. Moreover, for a quantum exterior
algebra $\Gamma$ on $n$ variables (that is, a codimension $n$
quantum complete intersection where all the defining exponents are
$2$), the global dimension of the endomorphism ring of the graded
generator-cogenerator $\oplus \Gamma / ( \Rad \Gamma )^i$ is $n+1$
(cf.\ \cite{AuRD}). Using this and
Proposition~\ref{proposition.repdimadd}, we obtain the following
improvement of Remark~\ref{remark.2nupper}.

\begin{theorem}
If $h = \card \{i \mid a_i = 2\}$, then
\[ \repdim \Lambda \leq \left\{ \begin{matrix} 2n - h & \text{ if } h \leq
n/2 \\ 2n-h+1 & \text{ if } h > n/2. \end{matrix} \right.
\]
\end{theorem}

\begin{proof}
In the first case decompose the algebra into $h$ parts of the form
$k\left<x,y\right> / (yx-qxy, x^2, y^a)$, and $n - 2h$ parts of the
form $k[x] / (x^a)$. Adding up the global dimensions of the
endomorphism rings of the graded Auslander generators (which we may
do by Proposition~\ref{proposition.repdimadd}), we obtain $h \cdot 3
+ (n-2h) \cdot 2 = 2n - h$. In the second case, we decompose the
algebra into $n-h$ parts of the form $k\left<x,y\right> / (yx-qxy,
x^2, y^a)$, and a quantum exterior algebra on $2h-n$ variables, and
add up global dimensions as above.
\end{proof}

\bibliographystyle{amsplain}
\bibliography{sources}

\end{document}